\documentclass[12pt,final]{article}
\usepackage{amsmath,amsthm,amsfonts,amssymb,graphicx,hyperref,enumerate,psfrag}
\usepackage{tikz,pgfplots}
\usepackage{fullpage}

\usepackage[utf8]{inputenc} 
\usepackage[english]{babel}
\definecolor{metrorange}{RGB}{235,131,32}
\definecolor{myblue}{RGB}{51,51,178}
\definecolor{myred}{RGB}{189,26,26}
\definecolor{mygreen}{RGB}{0,128,0}

\newtheorem{definition}{Definition}
\newtheorem{lemma}{Lemma}

\newtheorem{remark}{Remark}

\newtheorem{theorem}{Theorem}
\newtheorem{conjecture}[lemma]{Conjecture}

\newtheorem{corollary}{Corollary}

\newcommand{\EE}{{\mathbb{E}}}
\newcommand{\dd}{\,{\rm d}}
\newcommand{\PP}{\mathbb{P}}

\newcommand{\N}{\mathbb {N}}
\newcommand{\R}{\mathbb {R}}

\newcommand{\cX}{\mathcal {X}}
\newcommand{\cP}{\mathcal {P}}

\newcommand{\dist}{\mathrm{dist}}
\newcommand{\dtv}{d_{\textsc{tv}}}
\newcommand{\dkl}{d_{\textsc{kl}}}
\newcommand{\vkl}{V_{\textsc{kl}}}
\newcommand{\pim}{p}
\newcommand{\tmix}{t_{\textsc{mix}}}
\newcommand{\diam}{\mathrm{diam}}

\newcommand{\lip}{{\textsc{lip}}}

\newcommand{\var}{{\mathrm{Var}}}

\title{The varentropy criterion is sharp on expanders}

\author{Justin Salez}

\begin{document}
\maketitle
\begin{abstract}
The cutoff phenomenon is an abrupt transition from out of equilibrium to equilibrium
undergone by certain Markov processes in the limit where the size of the state space tends
to infinity: instead of decaying gradually over time, their distance to equilibrium remains close
to the maximal value for a while and suddenly drops to zero as the time parameter reaches a
critical threshold. Despite the accumulation of many examples, this phenomenon is still far from being understood, and identifying the general conditions that trigger it has become one of the biggest challenges in the quantitative analysis of finite Markov chains. Very recently,  the author proposed a general sufficient condition for the occurrence of a cutoff, based on a certain information-theoretical statistics called \emph{varentropy}. In the present paper, we demonstrate the sharpness of this approach by showing that the cutoff phenomenon is actually \emph{equivalent} to the varentropy criterion  for all sparse, fast-mixing chains. Reversibility is not required. 
\end{abstract}
\section{Introduction}

The cutoff phenomenon is a dynamical phase transition which is now believed to be universal among high-dimensional, fast-mixing Markov chains: roughly speaking, the system under consideration abruptly moves from being nearly singular to equilibrium to being statistically indistinguishable from equilibrium when the time parameter reaches a critical value. We shall here only recall the necessary definitions, and refer the reader to the introductory book \cite[Chapter 18]{MR3726904} or the recent paper \cite{arxiv.2102.05597} for a more detailed account  as well as many references.

\subsection{The cutoff phenomenon}
Consider a stochastic matrix $K$ on a finite set $\cX$, and write $(P_t)_{t\ge 0}$ for the  corresponding continuous-time semi-group, defined for all times $t\ge 0$ and states $x,y\in\cX$ by
\begin{eqnarray*}
P_t(x,y) &  := & e^{-t}\sum_{n=0}^\infty \frac{K^n(x,y)t^n}{n!}.
\end{eqnarray*}
Assuming that $K$ is irreducible, the general theory guarantees that $P_t(x,y)\xrightarrow[t\to\infty]{}\pi(y)$, where  $\pi$ is the unique probability vector on $\cX$ solving the stationarity equation $\pi K=\pi$. A standard way to quantify this convergence consists in measuring the time it takes for the \emph{worst-case total variation distance}  to drop below a given threshold $\varepsilon\in(0,1)$:
\begin{eqnarray*}
\tmix(\varepsilon) \ := \ \min\left\{t\ge 0\colon \dtv(t)\le\varepsilon\right\}, & \textrm{ where } &   \dtv(t) := \max_{o\in\cX}\dtv(P_t(o,\cdot),\pi).
\end{eqnarray*}
This quantity is known as the \emph{mixing time} of the process, and understanding how it depends on the underlying dynamics and on the precision $\varepsilon$ is an important problem with numerous applications. This question becomes  particularly relevant when the number of states is large, and one is thus naturally led to consider a \emph{sequence} of stochastic matrices $(K_n)_{n\ge 1}$ whose dimensions tend to infinity, and to examine the asymptotic behavior of their mixing times $\tmix^{(n)}(\varepsilon)$ as  $n\to\infty$. 
 In many situations, a remarkable phase transition known as a \emph{cutoff} has been observed: instead of decaying gradually from $1$ to $0$ as one could reasonably expect, the distance to equilibrium $t\mapsto\dtv^{(n)}(t)$ approaches a step function as $n\to\infty$. Equivalently, its inverse $\varepsilon\mapsto \tmix^{(n)}(\varepsilon)$ becomes asymptotically constant, as illustrated on Figure \ref{fig:cutoff}. 
\begin{definition}[Cutoff phenomenon]The sequence $(K_n)_{n\ge 1}$ is said to exhibit  a cutoff if 
\begin{eqnarray*}
\forall \varepsilon,\varepsilon'\in(0,1),\qquad \frac{\tmix^{(n)}(\varepsilon')}{\tmix^{(n)}(\varepsilon)} & \xrightarrow[n\to\infty]{} & 1.
\end{eqnarray*}
\end{definition} 
\begin{figure}
\begin{center}
\pgfplotsset{compat=1.16}
\pgfplotsset{ticks=none}
\begin{tikzpicture}
	\begin{axis}[
		width = 16cm,
		height = 9cm,
		axis x line=middle,
		axis y line=middle,
		xlabel = $t$,
		clip = false,
		grid=both,
		grid style={dashed, line width=.5pt, draw=gray!10},
		xmode = normal,
		ymode = normal,
		line width = 1pt,
		legend cell align = left,
		legend style = {fill=none, at={(0.9,0.9)}, anchor = north east},
		yticklabel style={above left},
		anchor = north west,
		tickwidth={5pt},
		xtick align = outside,
		ytick align = outside,
		ymax = 1.05,
		ymin = 0,
		xmin = 0,
		xmax = 40,
		x axis line style=-,
		y axis line style=-,
		domain = 0:40,
		samples = 1000
		]
		\def\scale{3}
\addplot[solid,  myred, line width = 1.5 pt] {
-rad(atan(x-20))/rad(atan(20))/2 + 0.5	
	};
\legend
{
	${t \to \dtv(t)}$
}
\draw[dashed, color= myblue] (15,0) -- (15,0.9515);
\draw[dashed, color= myblue] (0,0.9515) -- (15,0.9515);
\draw[dashed, color= mygreen] (25,0) -- (25,0.0484723);
\draw[dashed, color= mygreen] (0,0.0484723) -- (25,0.0484723);	

\draw[color= myblue] (0,0.9515) -- (-0.5,0.9515) node[color = black, anchor=east] {$\textcolor{myblue}{\varepsilon}$};
\draw[color= myblue] (15,0) -- (15,-0.02) node[color=black, anchor=north] {$\textcolor{myblue}{\tmix(\varepsilon)}$};
\draw[color= mygreen] (25,0) -- (25,-0.02) node[color = black, anchor=north] {${\textcolor{mygreen}{\tmix(\varepsilon')}}$};
\draw[color= mygreen] (0,0.0484723) -- (-0.5,0.0484723) node[color = black,anchor=east] {$\textcolor{mygreen}{\varepsilon'}$};

\draw[color= black] (0,1.0) -- (-0.5,1.0) node[color = black,anchor=south east] {${1}$};

\end{axis}
	
\end{tikzpicture}
\caption{A typical plot of the distance to equilibrium over time. In the large size limit, the ratio $\frac{\tmix(\varepsilon')}{\tmix(\varepsilon)}$ approaches $1$ and the convergence to equilibrium becomes abrupt (cutoff).}
\label{fig:cutoff}
\end{center}
\end{figure}
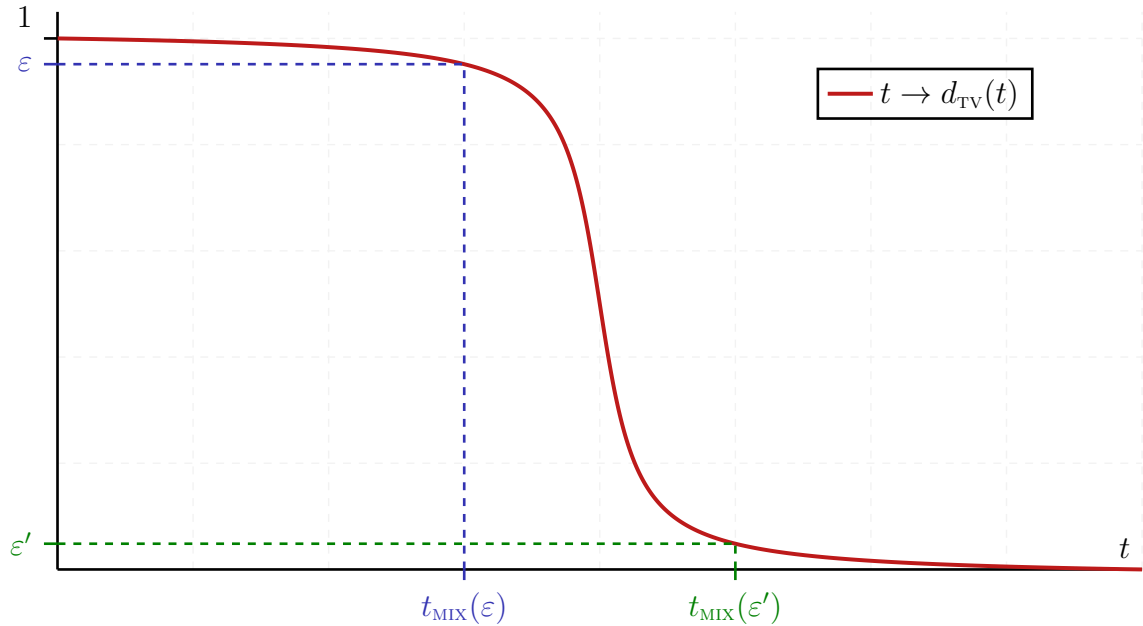 The name  \emph{cutoff} was coined   in 1986 by D. Aldous and P. Diaconis \cite{aldous1986shuffling}, but the phenomenon itself was actually discovered in 1981 by P. Diaconis and M. Shahshahani  \cite{diaconis1981generating}, and several instances of it were collected under the generic name \emph{abrupt switch} in lecture notes published by D. Aldous in 1983 \cite{aldous1983mixing}. Further historical examples can be found in the 1996 survey paper \emph{The cutoff phenomenon in finite Markov chains}, by P. Diaconis \cite{diaconis1996cutoff}.  Since then, cutoff phenomena have been observed  in a broad variety of contexts, including birth and death chains, random walks on finite groups, high-temperature spin glasses, interactive particle systems, or random walks on various models of sparse random graphs. 

Unfortunately, the existing proofs essentially all consist in bounding $\tmix^{(n)}(\varepsilon)$ from above and below by explicit quantities which lie within a factor $1+o(1)$ from each other and are asymptotically independent of $\varepsilon$. This is of course a notoriously difficult and model-specific task, which can only be carried out on very structured examples, and which does not bring any conceptual insight as to why a sharp transition actually occurs. Identifying the general conditions that trigger the cutoff phenomenon has become one of the biggest challenges in the quantitative analysis of finite Markov chains. Very recently, a new approach to this question was proposed in \cite{arxiv.2102.05597}, based on the estimation of a certain information-theoretical statistics called \emph{varentropy}.

\subsection{The varentropy criterion}

Let us start by recalling a more classical definition: the \emph{relative entropy} (or Kullback-Leibler divergence) of a probability measure $\mu$ on our reference space $(\cX,\pi)$  is given by
\begin{eqnarray*}
\dkl(\mu,\pi) & := & \EE_\mu\left[\log\frac{\mu}{\pi}\right] \ = \ \sum_{x\in\cX}\mu(x)\log\frac{\mu(x)}{\pi(x)}.
\end{eqnarray*}
This famous information-theoretic statistics provides an upper bound on the total-variation distance $\dtv(\mu,\pi)$, by virtue of the Csiszár-Kullback-Pinsker Inequality. Moreover, its evolution under the semi-group $(P_t)_{t\ge 0}$ can be controlled in a systematic way by establishing an appropriate Log-Sobolev Inequality \cite{MR1410112}, or its modified version \cite{MR2283379}. The combination of those two simple observations is at the origin of some of the most powerful  bounds on mixing times (see the textbook \cite{MR2341319}). In order to quantify the sharpness of the transition to equilibrium, the author proposed in \cite{arxiv.2102.05597} to investigate the evolution of a suitable \emph{second-order} version of relative entropy, obtained by replacing the mean by a variance:
\begin{eqnarray*}
\vkl(\mu,\pi) & := & \var_\mu\left(\log\frac{\mu}{\pi}\right)  \ = \ \sum_{x\in\cX}\mu(x)\left(\log\frac{\mu(x)}{\pi(x)}-\dkl(\mu,\pi)\right)^2.
\end{eqnarray*}
Because it measures the dispersion of information around the entropy, this natural statistics is   called \emph{varentropy}. It appeared a decade ago in the completely different context of optimal data compression, to quantify the error in the celebrated \emph{Asymptotic Equipartition Property} \cite{inproceedings}. However, its   relevance for cutoff -- embodied in Corollary \ref{co:main} below -- was discovered only very recently. More precisely, let us define the \emph{worst-case varentropy} of our Markov chain at any given time $t\ge 0$ as follows:
\begin{eqnarray*}
\vkl(t) & := & \max_{o\in\cX}\vkl\left(P_t(o,\cdot),\pi\right).
\end{eqnarray*}
Let also $\gamma=\gamma(K)$ denote the \emph{Poincaré constant} of the chain, which is well known to coincide with the spectral gap of the reversibilized transition matrix $(K+K^\star)/2$. It is perhaps worth mentioning that this fundamental parameter is, unlike many others, extremely well understood: its order of magnitude is known in many concrete models  (see \cite{MR2341319} for details). 
\begin{theorem}[Width of the mixing window \cite{arxiv.2102.05597}]\label{th:varentropy}
For any $\varepsilon \in(0,1/2)$,  
\begin{eqnarray*}
\tmix(\varepsilon)-\tmix(1-\varepsilon) & \le & \frac{2}{\gamma \varepsilon^2}\left(1+\sqrt{\vkl(\tmix(1-\varepsilon))}\right).
\end{eqnarray*}
\end{theorem}
To the best of our knowledge, Theorem \ref{th:varentropy} constitutes the very first general quantitative estimate on the width of the mixing window. 
Since the occurrence of a cutoff is just the assertion that this width is asymptotically negligible compared to the position of the window along the time axis, we readily obtain the following general  criterion for cutoff. 
\begin{corollary}[Varentropy criterion]\label{co:main}
A sufficient condition for $(K_n)_{n\ge 1}$ to exhibit cutoff is 
\begin{eqnarray}
\label{VC}
\gamma(K_n)\times \tmix^{(n)}(\varepsilon) & \gg & 1+\sqrt{\vkl^{(n)}(\tmix^{(n)}(\varepsilon))},
\end{eqnarray}
for each fixed $\varepsilon\in (0,1)$, where $a_n\gg b_n$ means that the ratio $a_n/b_n$ tends to $+\infty$ as $n\to\infty$. 
\end{corollary}
In the present form,  Corollary \ref{co:main} is much more a starting point  than a definitive answer to our main problem. Indeed, the varentropy term appearing on the right-hand side is a new and highly non-trivial statistics, whose  estimation remains entirely to be developed before it can be effectively  used to explain and predict cutoff. A first step in that direction was made in \cite{arxiv.2102.05597}, where a simple and naive estimate on the varentropy function $t\mapsto \vkl(t)$ was established for all Markov chains with non-negative curvature, leading to a unified proof of cutoff for a broad family of models. This successful first application raises hopes that the varentropy criterion could constitute the long-sought common mechanism underlying all cutoff phenomena. The purpose of this paper is to provide a rigorous support to this claim.

\subsection{Main result}
In the present paper, we demonstrate the sharpness of the varentropy approach by showing that the cutoff phenomenon is actually \emph{equivalent} to the varentropy criterion for all \emph{sparse} and \emph{fast-mixing} chains.  We emphasize that reversibility is not needed here: we shall  only require that the support of $K$ (i.e., the set of allowed transitions) is symmetric:
\begin{eqnarray}
\label{support}
\forall x,y\in\cX,\qquad K(x,y)>0 & \Longrightarrow & K(y,x)>0.
\end{eqnarray}
We recall that $\gamma=\gamma(K)$ denotes the Poincaré constant of $K$, and we define $\delta=\delta(K)$ as the smallest non-zero entry of $K$. This simple parameter controls the sparsity of the chain, since each row of the stochastic matrix $K$ can not have more than $1/\delta(K)$ non-zero entries.
\begin{theorem}[Sharpness of the varentropy criterion]\label{th:main}Let $(K_n)_{n\ge 1}$ be any sequence of transition matrices with symmetric supports and satisfying the following conditions:
\begin{enumerate}
\item[A1.] Sparsity: $\inf_{n\ge 1}\delta(K_n)>0$.
\item[A2.] Expansion: $\inf_{n\ge 1}\gamma(K_n)>0$.
\end{enumerate}
Then, the sequence $(K_n)_{n\ge 1}$ exhibits cutoff if and only if the varentropy criterion (\ref{VC}) holds. 
\end{theorem}

\begin{remark}[Cheeger inequalities] The \emph{isoperimetric constant} of the chain is defined as
\begin{eqnarray*}
\Phi  & := & \min_{\emptyset\subsetneq A\subsetneq \cX }\left\{\frac{\vec{\pi}(A\times A^c)}{\pi(A)\wedge \pi(A^c)}\right\},
\end{eqnarray*}
where $\vec{\pi}(x,y):=\pi(x)K(x,y)$ is the  stationary flow  on $\cX^2$. Cheeger inequalities state that
\begin{eqnarray*}
\frac{\Phi^2}{2} \ \le & \gamma & \le \ 2\Phi.
\end{eqnarray*}
 Consequently,  Assumption A2 is equivalent to $\inf_{n\ge 1}\Phi(K_n)>0$, hence the  name \emph{expansion}.
\end{remark}

Before we dive into the proof, let us briefly discuss the emblematic case of simple random walk on a finite undirected graph $G=(\cX,E)$, which corresponds to the transition matrix 
\begin{eqnarray*}
K(x,y) & :=& \left\{
\begin{array}{ll}
\frac{1}{\deg(x)} & \textrm{if }\{x,y\}\in E\\
0 & \textrm{else}.
\end{array}
\right.
\end{eqnarray*}
Note that the symmetry condition (\ref{support}) is automatically fulfilled here, and that $\delta(K)$ is simply the inverse of the maximum vertex degree.  Sequences of graphs whose transition matrices $(K_n)_{n\ge 1}$ satisfy Assumptions A1 and A2 are famously known as \emph{expanders}. Those remarkable graphs enjoy nearly as good connectivity properties as complete graphs, but at a much lower cost in terms of edges. Consequently, they have found numerous practical applications,  some of which are described in the beautiful survey paper \cite{MR2247919} by S. Hoory, N. Linial and A. Wigderson. Understanding when they exhibit cutoff is arguably one of the most famous open problems in the field  (see \cite[Open Question 34]{peresamerican} or \cite[Question 5]{MR3726904}), but to the best of our knowledge, no progress has been recorded beyond the  extreme case of {Ramanujan graphs} \cite{MR3558308,MR3693771,MR4178418,MR4476123} or the very special setup of random environments \cite{MR2667423,MR3758735,MR3650414,MR3773804,MR4132638,MR4385126,MR4476123,arxiv.2212.04469}. Our main result reduces this general problem to a varentropy estimate.
\begin{corollary}[Expanders]An expander sequence $(G_n)_{n\ge 1}$ exhibits cutoff if and only if
\begin{eqnarray*}
\forall\varepsilon\in(0,1),\qquad \vkl^{(n)}\left(\tmix^{(n)}(\varepsilon)\right) & \ll & (\log |G_n|)^2.
\end{eqnarray*}
\end{corollary}
We hope that this simple characterization will motivate the development of a general theory for estimating the varentropy of Markov chains. In particular, we would like to advertise the following fascinating conjecture, which was explicitly raised by D. Levin and Y. Peres \cite[Question 5]{MR3726904}. 
\begin{conjecture}[Transitive expanders]\label{pb:exp}All vertex-transitive expanders exhibit cutoff.
\end{conjecture}
We note that this is false without vertex-transitivity \cite{MR2774096}. Let us  perhaps here recall that a graph $G=(\cX,E)$  is \emph{vertex-transitive} if for any vertices $x,y\in \cX$, there is an edge-preserving bijection $\phi\colon \cX\to \cX$ that maps $x$ to $y$. In words, $G$ \emph{looks the same from every vertex}. This strong spatial homogeneity precludes many pathological phenomena observed in more heterogeneous settings, and entails considerably simplified expressions for a number of random-walk statistics \cite{MR994089,arxiv.2001.01467,MR4253426}. In light of this, it seems reasonable to hope that Conjecture \ref{pb:exp} will follow from a universal estimate on the varentropy of vertex-transitive expanders, and we intend to investigate this question in the near future.  

\section*{Acknowledgment}The author warmly thanks Gady Kozma and Jonathan Hermon for a stimulating discussion, as well as for their valuable comments on a preliminary version of the paper. This work was partly supported by Institut Universitaire de France.

\section{Proof}
Before we start, let us introduce some useful notation. First, we conveniently equip our state space $\cX$ with the following natural distance:
\begin{eqnarray*}
\dist(x,y) & := & \min\{n\in\N\colon K^n(x,y)>0\}.
\end{eqnarray*}
Note that the symmetry axiom is guaranteed by our assumption (\ref{support}), while the separation and triangle inequality are straightforward to check. This allows us to use the various notions pertaining to metric spaces. In particular, the \emph{diameter} of the state space is
\begin{eqnarray*}
\diam(\cX) & := & \max_{x,y\in\cX}\dist(x,y),
\end{eqnarray*}
while the \emph{Lipschitz norm} a function $f\colon\cX\to\R$ is given by:
\begin{eqnarray*}
\|f\|_\lip & := & \sup_{x\ne y}\frac{|f(x)-f(y)|}{\dist(x,y)}.
\end{eqnarray*}
We will also frequently use the following natural size parameter:
\begin{eqnarray*}
p & := & \min_{x\in\cX}\pi(x),
\end{eqnarray*}
which is always positive thanks to the irreducibility of $K$, but  tends to zero as the number of states grows. With this notation in hand, we may now recall two classical estimates on mixing times; see  \cite[Lemma 11]{arxiv.2102.05597} for the first and \cite[Corollary 2.6]{MR2341319} for the second. 
\begin{lemma}[Classical mixing-time estimates]\label{lm:tmix}For any $\varepsilon\in(0,1)$, we have
\begin{eqnarray*}
\tmix(\varepsilon) & \ge & \frac{1}{2}\diam(\cX)-\sqrt{\frac{2\tmix(\varepsilon)}{1-\varepsilon}}-\sqrt{\frac{2}{\gamma(1-\varepsilon)}};\\
\tmix(\varepsilon) & \le & \frac{1}{2\gamma}\log\left(\frac{1}{4\pim\varepsilon^2}\right).
\end{eqnarray*}
\end{lemma}
Our first observation is that under Assumptions A1-A2, those lower and upper bounds lie within a constant factor from each other, thereby providing explicit access to the exact order of magnitude of the mixing time.
\begin{lemma}[Control on $\pim$]\label{lm:pim}We always have
\begin{eqnarray*}
\log      \frac{1}{\pim}  & \le & 3\,\diam(\cX)\log \frac{1}{\delta} 
\end{eqnarray*}
\end{lemma}
\begin{proof}Fix $x,y\in\cX$ and set $n=\dist(x,y)$. We then have $K^n(x,y)>0$, and hence $K^n(x,y)>\delta^n$ by definition of $\delta$. Using the stationarity $\pi=\pi K=\cdots =\pi K^n$, we can then write
\begin{eqnarray*}
\pi(y) & = & \sum_{z\in\cX}\pi(z)K^n(z,y)\\
& \ge & \pi(x)\delta^{n}\\
& \ge & \pi(x)\delta^{\diam(\cX)}.
\end{eqnarray*}
Choosing $y$ so that $\pi(y)=\pim$ and summing over all $x\in\cX$, we obtain
\begin{eqnarray*}
\pim|\cX| & \ge & \delta^{\diam(\cX)}.
\end{eqnarray*}        
On the other hand, since the diagram of the chain has degrees at most $\delta^{-1}$, we have
\begin{eqnarray*}
|\cX| & \le & 1+\delta^{-1}+\delta^{-2}+\cdots+\delta^{-\diam(\cX)} \ \le \  \delta^{-2\diam(\cX)},
\end{eqnarray*}
because $\delta\le 1/2$. The claim is now readily obtained by combining the last two displays.
\end{proof}
We next recall a recent, general regularity estimate for the logarithm of the heat-kernel at any sufficiently large time $t\ge 0$. This is taken from \cite[Lemma 10]{arxiv.2102.05597}.
\begin{lemma}[Spatial regularity of the heat kernel]\label{lm:lip} For all $o\in\cX$ and $t\ge \diam(\cX)/4$, 
\begin{eqnarray*}
\left\|\log \frac{P_t(o,\cdot)}{\pi(\cdot)}\right\|_{\textsc{lip}} & \le & c := 3\log \frac{e}{\delta}. 
\end{eqnarray*}
\end{lemma}
We use this regularity to show that the relative entropy can not decrease too fast. More precisely, define the worst-case relative entropy to equilibrium at any time $t\ge 0$ as follows:
\begin{eqnarray*}
\dkl(t) & := & \max_{o\in\cX}\dkl(P_t(o,\cdot),\pi).
\end{eqnarray*}
\begin{lemma}[Regularity of relative entropy]\label{lm:slow}For any $t\ge \diam(\cX)/4$ and any $s\ge 0$, we have
\begin{eqnarray*}
 \dkl(t) & \le &  \dkl(t+s)+c s, 
\end{eqnarray*}
where $c$ is the constant appearing in the previous lemma.
\end{lemma}
\begin{proof}By an elementary and classical computation,  we have
\begin{eqnarray*}
-\frac{\dd}{\dd t}\,\dkl\left(P_t(o,\cdot),\pi\right) & = & \sum_{x,y\in\cX}P_t(o,x)K(x,y)\left(\log\frac{P_t(o,x)}{\pi(x)}-\log\frac{P_t(o,y)}{\pi(y)}\right)\\
& \le & \left\|\log \frac{P_t(o,\cdot)}{\pi(\cdot)}\right\|_{\textsc{lip}},
\end{eqnarray*}
and the claim now readily follows from Lemma \ref{lm:lip}.
\end{proof}
Another immediate consequence of Lemma \ref{lm:lip} is the following heat-kernel estimate.
\begin{lemma}[Uniform heat-kernel estimate]\label{lm:uniform}For all $o,x\in\cX$ and $t\ge \diam(\cX)/4$, we have
\begin{eqnarray*}
-c\,\diam(\cX) \ \le & \log\left(\frac{P_t(o,x)}{\pi(x)}\right) & \le \ \log\frac{1}{p},
\end{eqnarray*}
where $c$ is the constant appearing in Lemma \ref{lm:lip}.
\end{lemma}
\begin{proof}
The  first inequality is simply the crude bound $\max(h)-h(x)\le \diam(\cX)\|h\|_\lip$ applied to the function $h=\log \frac{P_t(o,\cdot)}{\pi}$, and the second  trivially follows from the definition of $p$. 
\end{proof}
Finally, we will need the following simple lemma, which asserts that the classical upper bound on $\dtv(t)$ using $\dkl(t)$ (Pinsker's inequality) can be reversed at a reasonable price.
\begin{lemma}[Reversed Pinsker's inequality]\label{lm:pinsker}For any $t\ge 0$, we have
\begin{eqnarray*}
\dkl(t) & \le & \left(\frac{1}{1-\pim}\log\frac 1{\pim}\right)\dtv(t).
\end{eqnarray*}
\end{lemma}
\begin{proof}
Since the function $g\colon u \mapsto \frac{u\log u }{u-1}$ is increasing on $[1,\infty)$, we have for all $0\le u\le v$,
\begin{eqnarray*}
u \log u & \le  & g(v)(u-1)_+.
\end{eqnarray*}
In particular, given $\mu\in\cP(\cX)$, we may take $u=\frac{\mu(x)}{\pi(x)}$ and $v=\frac{1}{\pim}$ to obtain
\begin{eqnarray*}
\frac{\mu(x)}{\pi(x)}\log \frac{\mu(x)}{\pi(x)} & \le & \left(\frac{\mu(x)}{\pi(x)}-1\right)_+g\left(\frac{1}{\pim}\right),
\end{eqnarray*}
for all $x\in\cX$. Averaging this with respect to $\pi$ yields
\begin{eqnarray*}
\dkl(\mu,\pi) & \le &\dtv(\mu,\pi) g\left(\frac{1}{\pim}\right).
\end{eqnarray*}
The claim now follows by specializing this to $\mu=P_t(o,\cdot)$ and maximizing over $o\in\cX$.
\end{proof}

We now have all we need to prove Theorem \ref{th:main}. 
\begin{proof}Let $(K_n)_{n\ge 1}$ be a sequence of transition matrices with symmetric support satisfying Assumptions A1-A2. Fix $\varepsilon\in(0,1)$ once and for all, and write $t_n:=\tmix^{(n)}(\varepsilon)$ and $p_n:=\min \pi_n$.  Combining Lemmas \ref{lm:tmix} and \ref{lm:pim}, we know that
\begin{eqnarray*}
t_n & \asymp & \log \frac{1}{p_n} \ \asymp \ \ \diam(\cX_n),
\end{eqnarray*}
where the notation $a_n\asymp b_n$ means that the ratio $a_n/b_n$ is bounded from above and below by positive constants that do not depend on $n$. We will repeatedly use this fact below, without notice. Now, assume that $(K_n)_{n\ge 1}$ exhibits cutoff. This guarantees the existence of a sequence of times $(s_n)_{n\ge 1}$ with the following properties:
\begin{eqnarray*}
\frac{s_n}{t_n} \xrightarrow[n\to\infty]{} 0, \qquad \textrm{ and }\qquad \dtv^{(n)}(t_n+s_n) \xrightarrow[n\to\infty]{} 0.
\end{eqnarray*}
In particular, Lemma \ref{lm:pinsker} ensures that as $n\to\infty$,
\begin{eqnarray*}
\dkl^{(n)}(t_n+s_n) & \ll & \log \frac{1}{p_n}. 
\end{eqnarray*}
Moreover, since $t_n\ge \diam(\cX_n)/4$ for large enough $n$ by Lemma \ref{lm:tmix}, we can safely invoke Lemma \ref{lm:slow} with $t=t_n$ and $s=s_n$ to deduce that we also have
\begin{eqnarray}
\label{half}
\dkl^{(n)}(t_n)  & \ll & \log \frac{1}{p_n}.
\end{eqnarray}
Now, choose an arbitrary initial state $o_n\in\cX_n$ for each $n\in\N$, and let $\mu_n:=P_{t_n}(o_n,\cdot)$ denote the distribution of the chain at time $t_n$ starting from $o_n$. Let $X_n$ denote a random variable with law $\mu_n$, and consider the random variable 
\begin{eqnarray*}
Z_n & := &  \frac{\mu_n(X_n)}{\pi_n(X_n)}.
\end{eqnarray*}
Note that we then have $\EE[Z_n^{-1}]=1$, $\EE[\log Z_n]=\dkl(\mu_n,\pi_n)$ and $\var(\log Z_n)=\vkl(\mu_n,\pi_n)$.
Let also $F\colon(0,\infty)\to[0,\infty)$ be the function defined by the formula
\begin{eqnarray*}
F(u) & := & \log u+\frac{1}{u}-1.
\end{eqnarray*}
This function is decreasing on $(0,1]$ and increasing on $[1,\infty)$, with $F(1)=0$. Thus, we may invoke Markov's inequality to deduce that for any fixed $\theta>0$,
\begin{eqnarray*}
\PP\left(Z_n\ge p_n^{-\theta}\right) & \le & \frac{\EE\left[F(Z_n)\right]}{F(p_n^{-\theta})} \ = \ \frac{\dkl(\mu_n,\pi_n)}{\theta\log \frac{1}{p_n} +p_n^{\theta}-1} \\
\PP\left(Z_n\le p_n^{\theta}\right) & \le & \frac{\EE\left[F(Z_n)\right]}{F(p_n^{\theta})} \ = \  \frac{\dkl(\mu_n,\pi_n)}{p_n^{-\theta}+\theta\log p_n-1}.
\end{eqnarray*}
The key point is that both estimates tend to $0$ as $n\to\infty$, thanks to (\ref{half}). In other words, we have established the following convergence in probability:
\begin{eqnarray}
\label{cvprob}
\frac{\log Z_n}{\log \frac{1}{p_n}} & \xrightarrow[n\to\infty]{\PP} & 0.
\end{eqnarray}
 To conclude, observe that by  Lemma \ref{lm:uniform}, the random variables $\left({\log Z_n}/{\log \frac{1}{p_n}}\right)_{n\ge 1}$ all take values in a fixed compact set. Thus, the convergence (\ref{cvprob}) automatically also holds in $L^2$. In particular, we may safely take variances on both sides to obtain
\begin{eqnarray*}
\sqrt{\vkl(\mu_n,\pi_n)} & \ll & \log\frac{1}{p_n} \ \asymp \ t_n.
\end{eqnarray*}
Since the initial state $o_n\in\cX_n$ was arbitrary, we may finally choose it so that $\vkl(\mu_n,\pi_n)=\vkl^{(n)}(t_n)$, and the result is proved.
\end{proof}

\bibliographystyle{plain}
\bibliography{draft}

\begin{thebibliography}{10}

\bibitem{aldous1983mixing}
David Aldous.
\newblock Random walks on finite groups and rapidly mixing {M}arkov chains.
\newblock In {\em Seminar on probability, {XVII}}, volume 986 of {\em Lecture
  Notes in Math.}, pages 243--297. Springer, Berlin, 1983.

\bibitem{MR994089}
David Aldous.
\newblock Hitting times for random walks on vertex-transitive graphs.
\newblock {\em Math. Proc. Cambridge Philos. Soc.}, 106(1):179--191, 1989.

\bibitem{aldous1986shuffling}
David Aldous and Persi Diaconis.
\newblock {Shuffling cards and stopping times}.
\newblock {\em American Mathematical Monthly}, pages 333--348, 1986.

\bibitem{MR4132638}
Anna Ben-Hamou.
\newblock A threshold for cutoff in two-community random graphs.
\newblock {\em Ann. Appl. Probab.}, 30(4):1824--1846, 2020.

\bibitem{MR3650414}
Anna Ben-Hamou and Justin Salez.
\newblock Cutoff for nonbacktracking random walks on sparse random graphs.
\newblock {\em Ann. Probab.}, 45(3):1752--1770, 2017.

\bibitem{MR3758735}
Nathana\"{e}l Berestycki, Eyal Lubetzky, Yuval Peres, and Allan Sly.
\newblock Random walks on the random graph.
\newblock {\em Ann. Probab.}, 46(1):456--490, 2018.

\bibitem{MR2283379}
Sergey~G. Bobkov and Prasad Tetali.
\newblock Modified logarithmic {S}obolev inequalities in discrete settings.
\newblock {\em J. Theoret. Probab.}, 19(2):289--336, 2006.

\bibitem{MR3773804}
Charles Bordenave, Pietro Caputo, and Justin Salez.
\newblock Random walk on sparse random digraphs.
\newblock {\em Probab. Theory Related Fields}, 170(3-4):933--960, 2018.

\bibitem{MR4476123}
Charles Bordenave and Hubert Lacoin.
\newblock Cutoff at the entropic time for random walks on covered expander
  graphs.
\newblock {\em J. Inst. Math. Jussieu}, 21(5):1571--1616, 2022.

\bibitem{MR1410112}
P.~Diaconis and L.~Saloff-Coste.
\newblock Logarithmic {S}obolev inequalities for finite {M}arkov chains.
\newblock {\em Ann. Appl. Probab.}, 6(3):695--750, 1996.

\bibitem{diaconis1996cutoff}
Persi Diaconis.
\newblock The cutoff phenomenon in finite {M}arkov chains.
\newblock {\em Proc. Nat. Acad. Sci. U.S.A.}, 93(4):1659--1664, 1996.

\bibitem{diaconis1981generating}
Persi Diaconis and Mehrdad Shahshahani.
\newblock {Generating a random permutation with random transpositions}.
\newblock {\em Probability Theory and Related Fields}, 57(2):159--179, 1981.

\bibitem{MR3693771}
Jonathan Hermon.
\newblock Cutoff for {R}amanujan graphs via degree inflation.
\newblock {\em Electron. Commun. Probab.}, 22:Paper No. 45, 10, 2017.

\bibitem{MR4385126}
Jonathan Hermon, Allan Sly, and Perla Sousi.
\newblock Universality of cutoff for graphs with an added random matching.
\newblock {\em Ann. Probab.}, 50(1):203--240, 2022.

\bibitem{arxiv.2212.04469}
Jonathan Hermon, Anđela Šarković, and Perla Sousi.
\newblock Cutoff for random walk on random graphs with a community structure,
  2022.

\bibitem{MR2247919}
Shlomo Hoory, Nathan Linial, and Avi Wigderson.
\newblock Expander graphs and their applications.
\newblock {\em Bull. Amer. Math. Soc. (N.S.)}, 43(4):439--561, 2006.

\bibitem{inproceedings}
Ioannis Kontoyiannis and Sergio Verdu.
\newblock Optimal lossless compression: Source varentropy and dispersion.
\newblock pages 1739--1743, 07 2013.

\bibitem{MR3726904}
David~A. Levin and Yuval Peres.
\newblock {\em Markov chains and mixing times}.
\newblock American Mathematical Society, Providence, RI, 2017.
\newblock Second edition of [ MR2466937], With contributions by Elizabeth L.
  Wilmer, With a chapter on ``Coupling from the past'' by James G. Propp and
  David B. Wilson.

\bibitem{MR3558308}
Eyal Lubetzky and Yuval Peres.
\newblock Cutoff on all {R}amanujan graphs.
\newblock {\em Geom. Funct. Anal.}, 26(4):1190--1216, 2016.

\bibitem{MR2667423}
Eyal Lubetzky and Allan Sly.
\newblock Cutoff phenomena for random walks on random regular graphs.
\newblock {\em Duke Math. J.}, 153(3):475--510, 2010.

\bibitem{MR2774096}
Eyal Lubetzky and Allan Sly.
\newblock Explicit expanders with cutoff phenomena.
\newblock {\em Electron. J. Probab.}, 16:no. 15, 419--435, 2011.

\bibitem{MR2341319}
Ravi Montenegro and Prasad Tetali.
\newblock Mathematical aspects of mixing times in {M}arkov chains.
\newblock {\em Found. Trends Theor. Comput. Sci.}, 1(3):x+121, 2006.

\bibitem{MR4178418}
Narutaka Ozawa.
\newblock An entropic proof of cutoff on {R}amanujan graphs.
\newblock {\em Electron. Commun. Probab.}, 25:Paper No. 77, 8, 2020.

\bibitem{peresamerican}
Y~Peres.
\newblock Aim research workshop on sharp thresholds for mixing times.
\newblock 2004.

\bibitem{arxiv.2102.05597}
Justin Salez.
\newblock Cutoff for non-negatively curved markov chains.
\newblock {\em Geom. Funct. Anal.}, 32, 2022.

\bibitem{arxiv.2001.01467}
Romain Tessera and Matthew Tointon.
\newblock Sharp relations between volume growth, isoperimetry and resistance in
  vertex-transitive graphs, 2020.

\bibitem{MR4253426}
Romain Tessera and Matthew C.~H. Tointon.
\newblock A finitary structure theorem for vertex-transitive graphs of
  polynomial growth.
\newblock {\em Combinatorica}, 41(2):263--298, 2021.

\end{thebibliography}
\end{document}